\documentclass[reqno]{amsart}
\usepackage{amsfonts}
\usepackage{amsmath,amsthm}
\usepackage{amsfonts,amssymb}
\usepackage{mathrsfs}
\usepackage{enumerate}
\usepackage[inner=3cm,outer=3cm]{geometry}
\allowdisplaybreaks

\newtheorem{thm}{Theorem}[section]
\newtheorem{cor}[thm]{Corollary}
\newtheorem{lem}[thm]{Lemma}
\newtheorem{prop}[thm]{Proposition}

\theoremstyle{remark}
\newtheorem{rem}[thm]{Remark}

\numberwithin{equation}{section}

\newcommand{\al}{\alpha}

\def\({\Bigl(}
\def \){ \Bigr)}

 \def\RR{{\mathbb R}}

\begin{document}
\def\RR{\mathbb{R}}
\def\Exp{\text{Exp}}
\def\FF{\mathcal{F}_\al}

\title[] {Hardy inequalities for fractional $(k,a)$-generalized harmonic oscillator}

\author[]{Wentao Teng}

\address{School of Science and Technology, Kwansei Gakuin University, Japan.}

\email{ wentaoteng6@sina.com}


\maketitle
\begin{center}
	\small{\textit{(To commemorate all the people who died of the coronavirus as a result of trade of wild animals in Wuhan seafood market)}}
\end{center}

\begin{abstract} In this paper, we will define $a$-deformed Laguerre operators $L_{a,\alpha}$ and $a$-deformed Laguerre holomorphic semigroups on $L^2\left(\left(0,\infty\right),d\mu_{a,\alpha}\right)$. Then we give a spherical harmonic expansion, which reduces to the Bochner-type identity when taking the boundary value $z=\frac{\pi i}2$, of the $(k,a)$-generalized Laguerre semigroup introduced by S. Ben Sa\"id, T. Kobayashi and B. \O rsted. And then we prove a Hardy inequality for fractional powers of the $a$-deformed Dunkl harmonic operator $\triangle_{k,a}:=\left|x\right|^{2-a}\triangle_k-\left|x\right|^a$ using this expansion. When $a=2$, the fractional Hardy inequality reduces to that of Dunkl-Hermite operators given by \'O. Ciaurri, L. Roncal and S. Thangavelu. The operators $L_{a,\alpha}$ also give a tangible characterization of the radial part of the $(k,a)$-generalized Laguerre semigroup on each $k$-spherical component $\mathcal H_k^m\left(\mathbb{R}^N\right)$ for $\lambda_{k,a,m}:=\frac{2m+2\left\langle
k\right\rangle+N-2}a\geq -1/2$ defined via decomposition of unitary representation.

\end{abstract}

\input amssym.def

\section{Introduction}

\

Dunkl theory is a far-reaching generalization of Euclidean Fourier analysis associated with root system with a rich structure parallel to ordinary Fourier analysis, where finite reflection groups play the role of orthogonal groups in Euclidean Fourier analysis. The Lebesgue measure was replaced by a weighted
measure $dm_k(x)=h_k(x)dx$ invariant under the reflection group and parameterized by a multiplicity function $k$, and the ordinary partial derivatives were replaced by a kind of differential-difference operators using the finite reflection groups and the multiplicity functions. Such differential-difference operators, called Dunkl operators, gave an explicit expression of the radial part of the Laplacian on a flat Riemann symmetric space. This theory has drawn considerable attention and there have been a lot of works on Dunkl's analysis in the last twenty years.

More recently, S. Ben Sa\"id, T. Kobayashi and B. \O rsted \cite{BSK2} gave a further far-reaching generalization of Dunkl theory by introducing a parameter $a>0$ arisen from the “interpolation” of the two $sl(2,\mathbb R)$ actions on the Weil representation of the metaplectic group $Mp(N,\mathbb R)$
and the minimal unitary representation of the conformal group $O(N+1,2)$. They deformed an $sl_2$ triple studied in \cite{B} via the parameter $a$ such that the $a$-deformed Dunkl harmonic oscillator $\triangle_{k,a}:=\left|x\right|^{2-a}\triangle_k-\left|x\right|^a$ is symmetric on the Hilbert space ${{L}^{2}}\left( {{\mathbb{R}}^{N}},{\vartheta_{k,a}}\left( x
\right)dx \right)$, where ${\vartheta_{k,a}}\left( x
\right)={{\left| x \right|}^{a-2}}h_k(x)$. In the case of 
$k\equiv 0$, such $a$-deformed harmonic oscillator is also a deformation of the operator $\left|x\right|\triangle-\left|x\right|$ studied by Kobayashi and Mano in \cite{KM1,KM2}. 
Motivated by the definition of the classical Fourier transform on ${{L}^{2}}\left( {{\mathbb{R}}^{N}}\right)$ given by Howe \cite{H} via classical harmonic oscillators,
they then
proved the existence of a $(k,a)$-generalized holomorphic semigroup $\mathcal I_{k,a}\left(z\right),\;\Re z\geq0$ with infinitesimal generator $\frac1a\triangle_{k,a}$ acting on ${{L}^{2}}\left( {{\mathbb{R}}^{N}},{\vartheta_{k,a}}\left( x
\right)dx \right)$. The $(k,a)$-generalized Laguerre semigroup  $\mathcal I_{k,a}\left(z\right)$ generalizes
the Hermite semigroup studied by Howe \cite{H} ($k\equiv0$ and $a=2$), the Laguerre semigroup
studied by Kobayashi and Mano \cite{KM1,KM2} ($k\equiv0$ and $a=1$), and the Dunkl Hermite semigroup studied by R\"osler \cite{R1} ($k\geq 0$, $a=2$ and $z=2t,\;t>0$). When taking the boundary value $z=\frac{\pi i}2$, the semigroup  $\mathcal I_{k,a}\left(z\right)$ reduces to the so-called $(k,a)$-generalized Fourier transform $F_{k,a}$, i.e.,\\
\begin{align}\label{Fka}F_{k,a}=c{\mathcal I}_{k,a}\left(\frac{\pi i}2\right),
\end{align}
where $c=e^{i\pi(\frac{2\left\langle k\right\rangle+N+a-2}{2a})}.$
The generalized Fourier transform includes
the Fourier transform ($k\equiv0$ and $a=2$), the Kobayashi-Mano Hankel transform ($k\equiv0$ and $a=1$), and the Dunkl transform \cite{Du2} ($k\geq0$ and $a=2$).

We will define a one-dimensional $a$-deformed Laguerre holomorphic semigroup $I_{a,\alpha;z}:=e^{-\frac za L_{a,\al}}$ with the infinitesimal generator $-\frac 1a L_{a,\al}$, where $L_{a,\al}$ is the $a$-deformed Laguerre operator, and show that it reduces to $a$-deformed Hankel transform $H_{a,\alpha}$ when taking the boundary value $z=\frac{\pi i}2$.  
The operators $L_{a,\al}$ also give an explicit expression of the radial part $\Omega_{k,a}^{(m)}\left(\gamma_z\right)$ of the $(k,a)$-generalized Laguerre semigroup on each $k$-spherical component $\mathcal H_k^m\left(\mathbb{R}^N\right)$ defined via decomposition of unitary representation in \cite[Section 4.1]{BSK2}, i.e., $\Omega_{k,a}^{(m)}\left(\gamma_z\right)f(s)=s^{m}I_{a,\lambda_{k,a,m};z}\left(\left(\cdot\right)^{-m}f\right)\left(s\right)$, $\Re z\geq 0$, $s>0$, where $\lambda_{k,a,m}:=\frac{2m+2\left\langle
k\right\rangle+N-2}a$, for $f\in L^2\left({\mathbb{R}}_+,r^{\;2\left\langle k\right\rangle+N+a-3}dr\right)$ and $\lambda_{k,a,m} \geq -1/2$ as will be shown in Section 5.
We denote $\lambda_a:=\frac{2\left\langle k\right\rangle+N-2}a$.
\begin{thm}
For any function $f\in L^2\left(\left(0,\infty\right),d\mu_{a,\alpha}\right)$, $\alpha \geq -1/2$, we have
$$e^{(\alpha+1)\pi i/2}I_{a,\alpha;\frac{\pi i}2}(f)=H_{a,\alpha}(f),$$
where the $a$-deformed Hankel transform is defined as 
$$H_{a,\alpha}(f)(r)=\frac1{a^\alpha\Gamma(\al+1)}\int_0^\infty f(s)j_\alpha\left(\frac2ar^\frac a2s^\frac a2\right)s^{a\alpha+a-1}\,ds$$
and $j_\alpha(t)=2^\alpha\Gamma(\alpha+1)t^{-\alpha}J_\alpha\left(t\right)$ is the normalized Bessel function. 
\end{thm}

We will then give a spherical harmonic expansion of the $(k,a)$-generalized Laguerre semigroup.
 \begin{thm}(Spherical harmonic expansion of the $(k,a)$-generalized Laguerre semigroup)
For $f\in {{L}^{2}}\left( {{\mathbb{R}}^{N}},{\vartheta_{k,a}}\left( x \right)dx \right)$, $4\left\langle k\right\rangle+2N+a-4 \geq 0$, and $x\in\mathbb{R}^N,\;x=rx'$, with $r\in\mathbb{R}^+$, $x'\in
\mathbb S^{N-1}$, we have
\begin{equation}\label{expansion2}{\mathcal I}_{k,a}\left(z\right)f\left(x\right)=\sum_{m,j}Y_{m,j}(x')r^mI_{a,\lambda_{k,a,m};z}\left(\left(\cdot\right)^{-m}f_{m,j}\right)\left(r\right),\end{equation}
where $\Re z\geq0$. Specially, the $(k,a)$-generalized Laguerre semigroup reduces to the one dimensional $a$-deformed Laguerre holomorphic semigroup for radial functions, that is,
for $f=f_0\left(\left|\cdot\right|\right)$, $f_0\in L^2\left({\mathbb{R}}_+,r^{\;2\left\langle k\right\rangle+N+a-3}dr\right)$ and $r=\left|x\right|$, we have
$${\mathcal I}_{k,a}\left(z\right)f\left(x\right)={\left({\mathcal I}_{k,a}\left(z\right)f\right)}_0\left(r\right),\;\;\;\;{\left({\mathcal I}_{k,a}\left(z\right)f\right)}_0\left(r\right)=I_{a,\lambda_{a};z}(f_0)\left(r\right).$$
\end{thm}
\begin{rem}
i). This theorem, together with \eqref{Fka} and Theorem 1.1, imply the Bochner-type identity in \cite[Theorem 5.21]{BSK2}, which was used in \cite{GI2} for Schwartz functions to prove Pitt's inequalities for the generalized Fourier transform. That is, taking the boundary value $z=\frac{\pi i}2$, the expansion reduces to 
\begin{align}\label{Fkasum}F_{k,a}f(x)=\sum_{m,j}e^{-i\pi m/a}Y_{m,j}(x')r^mH_{a,\lambda_{k,a,m}}\left(\left(\cdot\right)^{-m}f_{m,j}\right)\left(r\right).\end{align}
This theorem also generalizes the result in \cite{T2} that Hermite semigroups reduce to Laguerre semigroups of type $\frac N2-1$ (the case of $a=2$ and $k=0$) for radial functions on $\mathbb R^N$.\\
ii). When $a=2$ and $z=2t,\;t>0$, the expansion reduces to the formula given in Theorem 4.5 in \cite{CR}, but our proof is different from that in \cite{CR} even in this case because we used the new tools introduced by  S. Ben Sa\"id, T. Kobayashi and B. \O rsted \cite{BSK2} in the development of $(k,a)$-generalized Fourier analysis.
\end{rem}

We will be interested in Hardy inequalities of the form
\begin{equation}\label{1}\int_X\frac{\left|f\left(x\right)\right|^2}{\left(1+\left|x\right|^2\right)^\sigma}d\eta\left(x\right)\leq B_\sigma\left\langle L^\sigma f,f\right\rangle\end{equation}
(or the Hardy inequality with homogeneous potential) for given
$0<\sigma<1$, where $L^\sigma$ is the fractional powers of a non-negative
self-adjoint operator $L$ and $B_\sigma$ is a constant. It is a generalization of the classical Hardy inequality on $\mathbb{R}^N$ 
$$\frac{{(N-2)}^2}4\int_{\mathbb{R}^N}\frac{\left|f\left(x\right)\right|^2}{\left|x\right|^2}dx\leq\int_{\mathbb{R}^N}\left|\nabla f(x)\right|^2dx,\;N\geq3.$$

In \cite{CR}, \'O. Ciaurri, L. Roncal and S. Thangavelu worked with \textit{conformally invariant} fractional powers of Dunkl--Hermite operators ${\mathbf
H}_k=-\triangle_k+\left|x\right|^2$, where $\triangle_k$ is a generalization of the classical Laplacian on Euclidean space called Dunkl Laplacian, and proved the fractional Hardy
inequalities for these operators of form \eqref{1} using ground state
representation. The conformal invariant fractional powers was borrowed from the context of sublaplacians on Heisenberg groups (see \cite{Ron}). They also deduced the Hardy inequalities for pure fractional powers of Dunkl--Hermite operators $\mathbf H_k^\sigma$ (see \cite[Corollary 1.5]{CR}) as a consequence of the conformally invariant fractional Hardy inequalities.

We will prove a Hardy inequality of type \eqref{1}
for fractional powers of the $a$-deformed Dunkl--Hermite operator
$\triangle_{k,a}=\left|x\right|^{2-a}\triangle_k-\left|x\right|^a$ using the spherical harmonic expansion of the $(k,a)$-generalized Laguerre semigroup \eqref{expansion2}. 

\begin{thm}\label{thm2}
Let us define the constant
$$B_{\alpha,\sigma}^\delta:=\delta^\sigma\frac{\Gamma\left(\frac{\alpha+2+\sigma}2\right)}{\Gamma\left(\frac{\alpha+2-\sigma}2\right)}.$$
For $0<\sigma<1,\;\delta>0$ and $4\left\langle k\right\rangle+2N+a-4 \geq 0$,
$$\left(\frac a2\right)^\sigma B_{\lambda_a,\sigma}^\delta\int_{\mathbb{R}^N}^{}\frac{\left|f(x)\right|^2}{\left(\delta+{\displaystyle\frac2a}\left|x\right|^a\right)^\sigma}\vartheta_{k,a}(x)dx\leq{\left\langle\left(-\triangle_{k,a}\right)_{\sigma}f,f\right\rangle}_{L^2\left(\mathbb{R}^N,\vartheta_{k,a}(x)dx\right)}$$
for all $f\in C_0^\infty(\mathbb{R}^N)$.
\end{thm}

When $a=2$, this inequality reduces to the fractional Hardy inequality in \cite{CR}, which was proved using Dunkl--Hermite expansions. 
The definition of the modified fractional operator $\left(-\triangle_{k,a}\right)_{\sigma}$ will be given analogously as in \cite{CR} in Section 4. We can also deduce the Hardy inequalities for pure fractional powers of the operator $\left(-\triangle_{k,a}\right)^{\sigma}$ analogous to Corollary 1.5 in \cite{CR} from this Hardy inequality. An uncertainty principle for fractional powers of $\triangle_{k,a}$ can also be deduced from this Hardy inequality as in \cite{Ron}.

There have also been several other studies of Hardy inequalities of form \eqref{1}. For example, D. Gorbachev, V. Ivanov and S. Tikhonov \cite{GI1} proved a
sharp Pitt's inequality for Dunkl transform in
$L^2\left(\mathbb{R}^N\right)$. Such Pitt's inequalities can imply a Hardy inequality of the form \eqref{1} for fractional
powers of the Dunkl Laplacian $\triangle_k$. They also proved a
sharp Pitt's inequality for the generalized Fourier transform
$F_{k,a}$ in \cite{GI2} using the Bochner-type identity \eqref{Fkasum}, a particular case of the expansion \eqref{expansion2} we will use. By the formula (5.6 b) in \cite{BSK2},
The fractional powers of $-{{\left| x \right|}^{2-a}}{{\Delta
}_{k}}$ can be naturally defined as follows,
$$F_{k,a}\left(\left(-\left|\cdot\right|^{2-a}\Delta_k\right)^\beta f\right)\left(\xi\right)=\left(\left|\xi\right|^a\right)^\beta F_{k,a}\left(f\right)\left(\xi\right).$$ 
And then from the inversion formula \cite[Theorem 5.3]{BSK2} of the $(k,a)$-generalized Fourier transform, the Pitt's inequality in \cite{GI2} implies also a
Hardy inequality of the form \eqref{1} for $L=-{{\left| x \right|}^{2-a}}{{\Delta
}_{k}}$ for $a={\textstyle\frac2n},\;n\in{\mathbb{N}}_+$. When
$a=2$, this Hardy inequality reduces to that for fractional powers of the Dunkl Laplacian in \cite{GI1}. The two Pitt's inequalities imply the logarithmic
uncertainty principle for the Dunkl transform and $F_{k,a}$,
respectively.

This paper is organized as follows. In Section 2, we recall the
tools and concepts we will use to prove the main theorems.
We refer to
\cite{BSK1,BSK2,Du1} for the tools and concepts. 
In Section 3, we give the definitions of the $a$-deformed Laguerre convolution and the fractional $a$-deformed Laguerre
operators, and then
prove a Hardy inequality for the fractional $a$-deformed Laguerre
operators, which reduces to the Hardy inequality for fractional
Laguerre operators given in \cite{CR} when $a=2$. We will prove Theorem 1.1 in Section 3 as well. 
In Section 4 we give the proof of Theorem 1.2 using the tools introduced by S. Ben Sa\"id, T. Kobayashi and B. \O rsted \cite{BSK2} and then prove Theorem 1.4.
In Section 5 we will give the tangible characterization of the radial part of the $(k,a)$-generalized Laguerre semigroup on each $k$-spherical component $\mathcal H_k^m\left(\mathbb{R}^N\right)$ for $\lambda_{k,a,m} \geq -1/2$.

\section{Preliminaries}
\subsection{Dunkl operators and Dunkl transform}

\

Given a root system $R$ in the Euclidean space $\mathbb R^N$, denote by $G$ the
finite subgroup of $O(N)$ generated by the reflections
$\sigma_\alpha$ associated to the root system. Define a $multiplicity$ $function$
$k:R\rightarrow \mathbb C$ such that $k$ is $G$-invariant, that is,
$k\left(\alpha\right)=k\left(\beta\right)$ if $\sigma_\alpha$ and
$\sigma_\beta$ are conjugate. We assume $k\geqslant0$ in this
paper. The $Dunkl\;operators\;T_\xi$, $\xi\in\mathbb{R}^N$, which
were introduced in \cite{Du1}, are defined by the following
deformations by difference operators of directional derivatives
$\partial_\xi$: \begin{align*}T_\xi f(x)=\partial_\xi
f(x)+\sum_{\alpha\in
R^+}k(\alpha)\left\langle\alpha,\;\xi\right\rangle\frac{f(x)-f
(\sigma_\alpha(x))}{\left\langle\alpha,\;x\right\rangle},\end{align*}
where $R^+$ is any fixed positive system of $R$. They commute
pairwise and are skew-symmetric with respect to the $G$-invariant
measure $dm_k(x)=h_k(x)dx$, where
$$h_k(x)=\prod_{\alpha\in R}\vert\left\langle\alpha,\;x\right\rangle\vert^{k(\alpha)}.$$
Denote by $\mathbf N\boldsymbol=N+2\left\langle k\right\rangle$
the homogeneous dimension of the root system, where $\left\langle
k\right\rangle:=\sum_{\alpha\in R^+}k(\alpha)$. Let
$e_j,\;j=1,2,...,N,$ be the canonical orthonormal basis in
$\mathbb{R}^N$ and denote $T_j=T_{e_j}.$ The $Dunkl$ $Laplacian$
is defined by $\triangle_k={\textstyle\sum_{j=1}^N}T_j^2$ and it can be expressed explicitly.

The Dunkl
kernel $E(x,y)$ is the unique analytic solution to the
differential-difference equation system
$$T_\xi f=\left\langle\xi,y\right\rangle f,\;\;f(0)=1,$$ for any fixed $y\in\mathbb{R}^N$.
For $f\in L^1(m_k)$ the Dunkl transform is defined by
$$F(f)(\xi)=\frac1{c_k}\int_{\mathbb{R}^N}f(x)E(-i\xi,\;x)dm_k(x),\;c_k=\int_{\mathbb{R}^N}e^{-\frac{\left|x\right|^2}2}dm_k(x).$$
It is a generalization of and has similar properties with the
classical Fourier transform.

\subsection{An orthonormal basis in ${{L}^{2}}\left( {{\mathbb{R}}^{N}},{\vartheta_{k,a}}\left( x \right)dx \right)$}
\

An $h$-harmonic polynomial of degree $m$ is a homogeneous
polynomial $p$ on $\mathbb{R}^N$ of degree $m$ satisfying
$\triangle_kp=0$. Denote by $\mathcal
H_k^m\left(\mathbb{R}^N\right)$ the space of $h$-harmonic
polynomials of degree $m$. Spherical $h$-harmonics (or just $h$-harmonics) of degree $m$ are then
defined as the restrictions of $\mathcal
H_k^m\left(\mathbb{R}^N\right)$ to the unit sphere $\mathbb S^{N-1}$. The spaces $\mathcal
H_k^m\left(\mathbb{R}^N\right)\vert_{\mathbb S^{N-1}},\;m=0,1,2,\dots$ are finite
dimensional and orthogonal to each other with respect to the measure
$h_k\left(x'\right)d\sigma(x')$. And there is the spherical harmonics
decomposition
\begin{equation}\label{decomposition}L^2\left(\mathbb S^{N-1},h_k\left(x'\right)d\sigma(x')\right)=\sum_{m\in \mathbb N}^\oplus\mathcal H_k^m\left(\mathbb{R}^N\right)\vert_{\mathbb S^{N-1}}.\end{equation}

Consider the weight function ${\vartheta_{k,a}}\left( x
\right)={{\left| x \right|}^{a-2}}h_k(x)$. It reduces to
$h_k(x)$ when $a=2$ and for any $x'\in \mathbb S^{N-1}$, $${\vartheta_{k,a}}\left( x'
\right)=h_k\left(x'\right).$$ For the polar coordinates $x=rx'(r>0,\;x'\in
\mathbb S^{N-1})$,
$$\vartheta_{k,a}\left(x\right)dx=r^{2\left\langle k\right\rangle+N+a-3}\vartheta_{k,a}\left(x'\right)drd\sigma(x').$$

From the spherical harmonic decomposition \eqref{decomposition} of
$L^2\left(\mathbb S^{N-1},h_k\left(x'\right)d\sigma(x')\right)$, there is
a unitary isomorphism (see \cite[(3.25)]{BSK2})
$$\sum_{m\in \mathbb N}^\oplus(\mathcal H_k^m\left(\mathbb{R}^N\right){\vert_{\mathbb S^{N-1}})\otimes L^2}\left({\mathbb{R}}_+,r^{\;2\left\langle k\right\rangle+N+a-3}dr\right)\xrightarrow\sim L^2\left(\mathbb{R}^N,\vartheta_{k,a}\left(x\right)dx\right).$$

Define the Laguerre polynomial as
$$L_l^\mu(t):=\sum_{j=0}^l\frac{{(-1)}^j\Gamma(\mu+l+1)}{(l-j)!\Gamma(\mu+j+1)}\frac{t^j}{j!},\;\mathrm{Re}\mu>-1.$$

\begin{prop}\label{basis}(\cite[Proposition 3.15]{BSK2})
For fixed $m\in\mathbb{N}$, $a>0$, and a multiplicity function $k$
satisfying $\lambda_{k,a,m}>-1$. Set
\begin{equation}\label{psi}\psi_{l,m}^{(a)}(r):=\left(\frac{2^{\lambda_{k,a,m}+1}\Gamma(l+1)}{a^{\lambda_{k,a,m}}\Gamma(\lambda_{k,a,m}+l+1)}\right)^{1/2}r^mL_l^{\lambda_{k,a,m}}\left(\frac2ar^a\right)\mathrm{exp}\left(-\frac1ar^a\right).\end{equation}
Then $\left\{\psi_{l,m}^{(a)}(r):\;l\in\mathbb{N}\right\}$ forms
an orthonormal basis in
$L^2\left({\mathbb{R}}_+,r^{\;2\left\langle
k\right\rangle+N+a-3}dr\right)$.
\end{prop}

For each fixed $m\in\mathbb{N}$, we take an orthonormal basis of $\mathcal H_k^m\left(\mathbb{R}^N\right)\vert_{\mathbb S^{N-1}}$ as
\begin{equation}\label{Y}\left\{Y_i^m:\;i=1,2,\cdots,d(m)\right\},\end{equation}
where
$d(m)=\dim\left(\mathcal H_k^m\left(\mathbb{R}^N\right)\vert_{\mathbb S^{N-1}}\right)$. They are the eigenvectors of the generalized Laplace--Beltrami
operator $\triangle_{k;0}$. Proposition \ref{basis}
yields the orthonormal basis in ${{L}^{2}}\left(
{{\mathbb{R}}^{N}},{\vartheta_{k,a}}\left( x \right)dx \right)$
immediately.
\begin{cor}(\cite[Corollary 3.17]{BSK2})
Suppose $a>0$ and $k$ satisfy that $2m+2\left\langle
k\right\rangle+N+a-2>0$, Set
\begin{equation}\label{Phi}\mathrm\Phi_{l,m,j}^{(a)}\left(x\right):=Y_j^m\left(\frac x{\left|x\right|}\right)\psi_{l,m}^{(a)}\left(\left|x\right|\right).\end{equation}
Then \begin{equation}\label{basis4}\left\{\left.\mathrm\Phi_{l,m,j}^{(a)}\right| l\in\mathbb{N},\;m\in\mathbb{N},\;j=1,2,\cdots,d(m)\right\}\end{equation}
forms an orthonormal basis of ${{L}^{2}}\left(
{{\mathbb{R}}^{N}},{\vartheta_{k,a}}\left( x \right)dx \right)$.
\end{cor}

\subsection{The $(k,a)$-generalized Laguerre semigroup and Fourier transform}
\

Denote by
$$W_{k,a}\left(\mathbb{R}^N\right):=\mathbb{C}\text{-}span\left\{\left.\mathrm\Phi_{\mathit l}^{\mathit(\mathit a\mathit)}\mathit{(p,\cdot)}\right|l\in\mathbb{N},\;m\in\mathbb{N},\;p\in\mathcal H_k^m\left(\mathbb{R}^N\right)\right\},$$ where \begin{equation*}\label{p}\mathrm\Phi_l^{(a)}\left(p,x\right)=p(x')r^mL_l^{\lambda_{k,a,m}}\left(\frac2ar^a\right)\mathrm{exp}\left(-\frac1ar^a\right)\end{equation*} for $x=rx'\;(r>0,\;x'\in
\mathbb S^{N-1})$.
It
is a dense subset of the Hilbert space ${{L}^{2}}\left(
{{\mathbb{R}}^{N}},{\vartheta_{k,a}}\left( x \right)dx \right)$.
Define the $a$-deformed Dunkl-type harmonic oscillator with domain
$W_{k,a}\left(\mathbb{R}^N\right)$ as follows (see \cite{BSK1,BSK2}),
$${{\Delta }_{k,a}}={{\left| x \right|}^{2-a}}{{\Delta }_{k}}-{{\left| x \right|}^{a}},a>0.$$
It is an essentially self-adjoint operator on ${{L}^{2}}\left(
{{\mathbb{R}}^{N}},{\vartheta_{k,a}}\left( x \right)dx \right)$ with only negative discrete spectrum. 
And so $\frac1a\triangle_{k,a}$ is the infinitesimal generator of the $(k,a)$-generalized Laguerre semigroup ${\mathcal I}_{k,a}\left(z\right):=\exp \left(\frac za\triangle_{k,a}\right),\;\Re z\geq0$.
The semigroup ${\mathcal I}_{k,a}\left(z\right)$ can also be defined by a unitary representation, i.e., ${\mathcal I}_{k,a}\left(z\right):=\Omega_{k,a}^{}\left(\gamma_z\right),\;\Re z\geq0$ (see \cite{BSK2} for the detailed definition of $\Omega_{k,a}^{}\left(\gamma_z\right)$).
By Schwartz kernel theorem,  ${\mathcal I}_{k,a}\left(z\right)$ has an integral
representation by means of a distribution kernel $\Lambda_{k,a}\left(x,y;z\right)$.
We refer to \cite{BSK2} for the details on the distribution kernel.
The boundary value $z=\frac{\pi i}2$ of ${\mathcal I}_{k,a}\left(z\right)$ gives the definition of the $(k,a)$-generalized Fourier transform 
$F_{k,a}$. The
operator ${{F}_{k,a}}$ is a bijective linear operator such that
the Plancherel formula holds for $f\in {{L}^{2}}\left(
{{\mathbb{R}}^{N}},{\vartheta_{k,a}}\left( x \right)dx \right)$.

\section{A Hardy inequality for the fractional $a$-deformed Laguerre Operator}

\

The Laguerre translation $\mathcal T_r^{\alpha}$ was introduced by McCully \cite{Mc}
for $\alpha=0$ and was extended to $\alpha \geq -1/2$ (see \cite{A} or \cite[Chapter 6]{T1}). 
We define the \textit{a-deformed Laguerre translation} as
\begin{align*}\mathcal T_r^{a,\alpha}f(s):=&\frac{\Gamma(\alpha+1)2^\alpha}{\sqrt{2\pi}}\int_0^\pi f\left(\left(r^a+s^a+2r^{{}^\frac a2}s^{{}^\frac a2}\cos\theta\right)^{1/a}\right)\cdot\\& J_{\alpha-1/2}\left(\frac2ar^{{}^\frac a2}s^{{}^\frac a2}\sin\theta\right)\left(\frac2ar^{{}^\frac a2}s^{{}^\frac a2}\sin\theta\right)^{-(\alpha-1/2)}\left(\sin\theta\right)^{2\alpha}\,d\theta\end{align*}
for $r,s>0$ and $\alpha \geq -1/2$, where $J_{\nu}$ is the Bessel function of order $\nu$.
When $a=2$, it reduces to the Laguerre translation $\mathcal T_r^{\alpha}$ in \cite{CR}.
The results in \cite{CR} are also valid for the critical case when $\alpha=-1/2$ since the definition of the Laguerre translation can be extended to this case.
If $f$ and $g$ are functions defined on $(0,\infty)$, the \textit{a-deformed Laguerre convolution} $f\ast_{a,\al}g$ is given by
\begin{equation}
\label{eq:LagConv}
f\ast_{a,\al}g(r)=\int_0^{\infty}\mathcal{T}_r^{a,\al}f(s)g(s)s^{a\al+a-1}\,ds.
\end{equation}
By changing variables
$$r=\left(\frac a2\right)^{1/a}r_1^{2/a},\;\;s=\left(\frac a2\right)^{1/a}s_1^{2/a}$$ 
and setting
$$f_1=f\left(\left(\frac a2\right)^{1/a}\left(\cdot\right)^{2/a}\right),\;\;g_1=g\left(\left(\frac a2\right)^{1/a}\left(\cdot\right)^{2/a}\right),$$ 
we have
\begin{align*}\int_0^\infty\mathcal T_r^{a,\alpha}f(s)g(s)s^{a\alpha+a-1}\,ds&=\left(\frac a2\right)^{\alpha+1}\int_0^\infty\mathcal T_{r_1}^\alpha f_1(s_1)g_1(s_1)s_1^{2\alpha+1}\,ds_1=\left(\frac a2\right)^{\alpha+1}f_1\ast_\alpha g_1(r_1)\\&=\left(\frac a2\right)^{\alpha+1}g_1\ast_\alpha f_1(r_1)=\left(\frac a2\right)^{\alpha+1}\int_0^\infty\mathcal T_{r_1}^\alpha g_1(s_1)f_1(s_1)s_1^{2\alpha+1}\,ds_1\\&=\int_0^\infty\mathcal T_r^{a,\alpha}g(s)f(s)s^{a\alpha+a-1}\,ds,\end{align*}
where $f\ast_{\al}g$ is the Laguerre convolution defined in \cite[Chapter 6]{T1}.  Thus $f\ast_{a,\alpha}g(r)=g\ast_{a,\alpha}f(r)$.

Let \begin{equation*}
\varphi_l^{a,\al}(r):=L_l^\alpha\left(\frac2ar^a\right)\mathrm{exp}\left(-\frac1ar^a\right),\quad l=0,1,\cdots.
\end{equation*}
Then substituting $r$ as $\sqrt{\frac2a}r^\frac a2$ and $s$ as $\sqrt{\frac2a}s^\frac a2$ in the formula (3.2) in \cite{CR}, we get
\begin{equation}
\label{eq:Lagtransvarphi}
\mathcal{T}_r^{a,\alpha} \varphi_n^{a,\al}(s)=\frac{n!}{(\alpha+1)_n}
\varphi_n^{a,\al}(r)\varphi_n^{a,\al}(s),\qquad \alpha \geq -1/2.
\end{equation}

The Laguerre operator 
\begin{equation}\label{Laguerre0}L_\alpha=-\frac{d^2}{dr^2}+r^2-\frac{2\alpha+1}r\frac d{dr}
\end{equation}
studied in \cite{CR} is a symmetric operator on $L^2\left(\left(0,\infty\right),d\mu_\alpha\right)$, where $\alpha \geq -1/2$ and $d\mu_\alpha\left(r\right)=r^{2\alpha+1}dr$. The functions 
$$\widetilde{\varphi}_l^\alpha(r)=\left(\frac{2\Gamma(l+1)}{\Gamma(\alpha+l+1)}\right)^{1/2}L_l^\alpha\left(r^2\right)\mathrm{exp}\left(-\frac12r^2\right),\;l=0,1,\cdots$$
are eigenfunctions of $L_\alpha$ with eigenvalues $2\left(2l+\alpha+1\right)$.

Substituting $r$ by $u=\sqrt{\frac2a}r^\frac a2$ in
\eqref{Laguerre0},
\begin{align*}-\frac{d^2}{du^2}+u^2-\frac{2\alpha+1}u\frac d{du}&=-\frac2a\left(\frac1{r^{a-2}}\frac{d^2}{dr^2}+\left(1-\frac a2\right)\frac1{r^{a-1}}\frac d{dr}\right)+\frac2ar^a-\frac{2\alpha+1}{r^{a-1}}\frac d{dr}\\&=\frac2a\left(-\frac1{r^{a-2}}\frac{d^2}{dr^2}+r^a-\left(a\alpha+1\right)\frac1{r^{a-1}}\frac d{dr}\right).
\end{align*}
The $a$-deformed Laguerre differential
operator can then be defined as
\begin{equation}\label{Laguerre}
L_{a,\alpha}=-\frac1{r^{a-2}}\frac{d^2}{dr^2}+r^a-\left(a\alpha+1\right)\frac1{r^{a-1}}\frac d{dr}.
\end{equation}
It is symmetric on $L^2\left(0,\infty\right)$ with respect to the
measure $d\mu_{a,\alpha}(r)=r^{a\alpha+a-1}dr$,  $\alpha \geq -1/2$. When $a=2$, the
operator reduces to the Laguerre operator \eqref{Laguerre0}.

Define the Laguerre functions of type $\alpha$ as
$$\widetilde{\varphi}_l^{a,\alpha}(r)=\left(\frac{2^{\alpha+1}\Gamma(l+1)}{a^\alpha\Gamma(\alpha+l+1)}\right)^{1/2}L_l^\alpha\left(\frac2ar^a\right)\mathrm{exp}\left(-\frac1ar^a\right),\;l=0,1,\cdots,$$
where $\alpha \geq -1/2$. Then they form an
orthonormal basis of
$L^2\left(\left(0,\infty\right),d\mu_{a,\alpha}\right)$ (this is also the case of Proposition 2.1 when $\alpha=\lambda_{k,a,m}$) and are the eigenfunctions of the $a$-deformed Laguerre operator
\eqref{Laguerre}. Indeed,
$$L_{a,\alpha}\widetilde{\varphi}_l^{a,\alpha}=a\left(2l+\alpha+1\right)\widetilde{\varphi}_l^{a,\alpha},\;\;l=0,1,\cdots.$$
It suffices to substitute $r$ by $\sqrt{\frac2a}r^\frac a2$ in the conclusions of \cite[Section 3]{CR} to get this.

The Laguerre expansion of $ f \in L^2((0,\infty), d\mu_{a,\al})$, namely the expansion
$$f=\sum_{l=0}^\infty\left(\frac{2^{\alpha+1}\Gamma(l+1)}{a^\alpha\Gamma(\alpha+l+1)}\right){\left\langle f,\varphi_l^{a,\alpha}\right\rangle}_{d\mu_{a,\alpha}}\varphi_l^{a,\alpha}$$ can be written in a compact form in terms of Laguerre convolution.
\begin{lem}
For a function $f \in L^2((0,\infty),d\mu_{a,\al}) $,  $\varphi_l^{a,\alpha}$ is an eigenfunction of $f$, i.e.
$$
f\ast_{a,\alpha}\varphi_l^{a,\alpha}=\frac{\Gamma(\alpha+1)\Gamma(l+1)}{\Gamma(\alpha+l+1)}{\left\langle f,\varphi_l^{a,\alpha}\right\rangle}_{d\mu_{a,\alpha}}\varphi_l^{a,\alpha}.
$$
In particular,
\begin{equation}
\label{eq:convoLaguerres}
\delta_{nj}\varphi_n^{a,\al}=\frac{2^{\alpha+1}}{a^\alpha\Gamma(\al+1)}\varphi_n^{a,\al}\ast_{a,\al}\varphi_j^{a,\al}.
\end{equation}
\end{lem}

\begin{proof}
Omitted. It is only a slight modification of the proof of Lemma 3.1 in \cite{CR}.
\end{proof}

Thus $f\ast_{a,\al}\varphi_l^{a,\al}$ are eigenfunctions of $L_{a,\alpha}$ with the eigenvalues $a\left(2l+\alpha+1\right)$ for $l=0,1,\cdots$ and we have the spectral decomposition of the 
$a$-deformed Laguerre operator
$$L_{a,\alpha}f=\frac{2^{\alpha+1}}{a^\alpha\Gamma(\alpha+1)}\sum_{l=0}^\infty a\left(2l+\alpha+1\right)f\ast_{a,\alpha}\varphi_l^{a,\alpha}.$$
It is then natural to define fractional powers of Laguerre operators as
$$L_{a,\alpha}^\sigma f=\frac{2^{\alpha+1}}{a^\alpha\Gamma(\alpha+1)}\sum_{l=0}^\infty\left(a\left(2l+\alpha+1\right)\right)^\sigma f\ast_{a,\alpha}\varphi_l^{a,\alpha},\quad \alpha \geq -1/2.$$
But it suits better to work with the modified fractional
operator $L_{a,\alpha;\sigma}$ with the spectrum $4^\sigma S_l^{a,\alpha;\sigma}$, i.e.
$$L_{a,\alpha,\sigma}f=\frac{2^{\alpha+1}}{a^\alpha\Gamma(\alpha+1)}\sum_{l=0}^\infty(2a)^\sigma S_l^{a,\alpha;\sigma}f\ast_{a,\alpha}\varphi_l^{a,\alpha},\quad \alpha \geq -1/2,$$
where
$$S_l^{a,\alpha;\sigma}=\frac{\Gamma\left(\frac{a\left(2l+\alpha+1\right)}{2a}+\frac{1+\sigma}2\right)}{\Gamma\left(\frac{a\left(2l+\alpha+1\right)}{2a}+\frac{1-\sigma}2\right)},$$
because such fractional powers of the operator correspond to the conformally invariant fractional powers of sublaplacian $\mathcal L$ on Heisenberg groups when we consider the conformally invariant fractional powers $\mathcal L_\sigma$ (see \cite{Ron}) acting on the functions of the form $e^{it}f\left(\left|z\right|\right)$.
In short, we write
$$L_{a,\alpha;\sigma}={(2a)}^\sigma\frac{\Gamma\left(\frac{L_{a,\alpha}}{2a}+\frac{1+\sigma}2\right)}{\Gamma\left(\frac{L_{a,\alpha}}{2a}+\frac{1-\sigma}2\right)}.$$
The motivation for this definition goes back to \cite[(1.33)]{Br}, for instance.

For $\delta>0$ and $\alpha \geq -1/2$, denote
$$\omega_{\alpha,\sigma}^{\delta,a}(r):=c_{\alpha,\sigma}\left(\delta+{\textstyle\frac2a}r^a\right)^{-(\alpha+1+\sigma)/2}K_{(\alpha+1+\sigma)/2}\left(\frac{\delta+\frac2ar^a}2\right),$$
where $K_\nu$ is the Macdonald's function of order $\nu$ (see
\cite[Chapter 5, Section 5.7]{L}), and $c_{\alpha,\sigma}$ is the
constant
$$c_{\alpha,\sigma}:=\frac{\sqrt{\mathrm\pi}2^{1-\sigma}}{\Gamma\left((\alpha+2+\sigma)/2\right)}.$$
In \cite{CR}, the authors proved a Hardy inequality for the
fractional Laguerre operator for the case of $a=2$ using ground state representation.
\begin{thm}(\cite[Theorem 1.1]{CR})
Let $0<\sigma<1,\;\delta>0$, and $2\alpha+1>0$. Then
$$B_{\alpha,\sigma}^\delta\int_0^\infty\frac{\left|f(r)\right|^2}{\left(\delta+r^2\right)^\sigma}d\mu_\alpha(r)\leq\frac{4^\sigma}{\delta^\sigma}\left(B_{\alpha,\sigma}^\delta\right)^2\int_0^\infty\left|f(r)\right|^2\frac{\omega_{\alpha,\sigma}^\delta(r)}{\omega_{\alpha,-\sigma}^\delta(r)}d\mu_\alpha(r)\leq{\left\langle L_{\alpha,\sigma}^{}f,f\right\rangle}_{d\mu_\alpha}$$
for all $f\in C_0^\infty(0,\infty)$.
\end{thm}
Taking $f$ as the Laguerre functions for the case of $a=2$, and then
substituting $r$ by $\sqrt{\frac2a}r^\frac a2$, we get
\begin{align*}B_{\alpha,\sigma}^\delta\int_0^\infty\frac{\left|\widetilde\varphi_l^{a,\alpha}(r)\right|^2}{\left(\delta+{\displaystyle\frac2a}r^a\right)^\sigma}d\mu_{a,\alpha}(r)\leq\frac{4^\sigma}{\delta^\sigma}\left(B_{\alpha,\sigma}^\delta\right)^2\int_0^\infty\left|\widetilde\varphi_l^{a,\alpha}(r)\right|^2\frac{\omega_{\alpha,\sigma}^{\delta,a}(r)}{\omega_{\alpha,-\sigma}^{\delta,a}(r)}d\mu_{a,\alpha}(r)\leq{\left\langle{\left(\frac2aL_{a,\alpha}\right)}_\sigma\widetilde\varphi_l^{a,\alpha},\widetilde\varphi_l^{a,\alpha}\right\rangle}_{d\mu_{a,\alpha}}\end{align*}
for $\alpha \geq -1/2$. 
Here ${\left(\frac2aL_{a,\alpha}\right)}_\sigma=4^\sigma\frac{\Gamma\left(\frac{\frac2aL_{a,\alpha}}4+\frac{1+\sigma}2\right)}{\Gamma\left(\frac{\frac2aL_{a,\alpha}}4+\frac{1-\sigma}2\right)}$ and it equals to $\left(\frac2a\right)^\sigma L_{a,\alpha;\sigma}$.

Then using the expansion via Laguerre
functions, we derive the Hardy inequality for the fractional $a$-deformed Laguerre operator.
\begin{thm}
Let $0<\sigma<1,\;\delta>0$, and $\alpha \geq -1/2$. Then
\begin{align*}\left(\frac a2\right)^\sigma B_{\alpha,\sigma}^\delta\int_0^\infty\frac{\left|f(r)\right|^2}{\left(\delta+{\displaystyle\frac2a}r^a\right)^\sigma}d\mu_{a,\alpha}(r)\leq\left(\frac{2a}\delta\right)^\sigma\left(B_{\alpha,\sigma}^\delta\right)^2\int_0^\infty\left|f(r)\right|^2\frac{\omega_{\alpha,\sigma}^{\delta,a}(r)}{\omega_{\alpha,-\sigma}^{\delta,a}(r)}d\mu_{a,\alpha}(r)\leq{\left\langle L_{a,\alpha;\sigma}f,f\right\rangle}_{d\mu_{a,\alpha}}\end{align*}
for all $f\in C_0^\infty(0,\infty)$.
\end{thm}

The holomorphic semigroup related to the $a$-deformed Laguerre operator
$L_{a,\alpha}$ is defined on
$L^2(\left(0,\infty\right),$
$d\mu_{a,\alpha})$ by
\begin{equation}
\label{eq:heatsemigroupLag}
I_{a,\al;z}f=e^{-\frac za L_{a,\al}}f,\quad \Re z\geq0.
\end{equation}
From the spectral decomposition of $L_{a,\al}$ it equals to $$\frac{2^{\alpha+1}}{a^\alpha\Gamma(\al+1)}\sum_{l=0}^{\infty}e^{-z(2l+\al+1)}
f\ast_{a,\al}\varphi_l^{a,\al}.$$

\

\noindent{\it Proof of Theorem 1.1}.

Define
\begin{align*}
q_{a,\al;z}(r):=\frac{2^{\alpha+1}}{a^\alpha\Gamma(\al+1)}\sum_{l=0}^{\infty}e^{-z(2l+\al+1)}\varphi_l^{a,\al}(r)=\left(\frac2a\right)^\alpha q_{2,\al;z}\left(\sqrt{\frac2a}r^\frac a2\right).
\end{align*}
Then we can write
$$
e^{-\frac zaL_{a,\al}}f=f\ast_{a,\al}q_{a,\al;z}.
$$
We give the kernel of the holomorphic semigroup $I_{a,\al;z}$.
\begin{lem}
\label{lem:Lagtranskernel}
Let $\alpha \geq -1/2$, $\Re z\geq 0$ and $z\neq 0$, we have that
$$
\mathcal{T}^{a,\al}_r q_{a,\al;z}(s)=\frac{e^{-\frac{\coth z}{a}(r^a+s^a)}}{(r^\frac a2s^\frac a2)^{\al} \sinh z}I_\al\left(\frac{\frac{2}{a}r^\frac a2s^\frac a2}{\sinh z}\right),
$$
where $I_{\al}$ is the modified Bessel function of the first
kind and order $\al$, see \cite[Chapter 5, Section 5.7]{L}.
\end{lem}
\begin{proof}
For the case when $a=2$, we take $w=e^{-z}$ in the equality (see \cite[p. 83]{T1})
\begin{equation*}
\sum_{n=0}^{\infty}\frac{\Gamma(n+1)}{\Gamma(n+\alpha+1)}\varphi_n^{\alpha}(r)
\varphi_n^{\alpha}(s)w^{2n}=(1-w^2)^{-1}(rsw)^{-\alpha}\exp\Big\{-\frac{1}{2}\Big (\frac{1+w^2}{1-w^2}\Big )
 (r^2+s^2)\Big\}I_{\alpha}\Big(\frac{2wrs}{1-w^2}\Big),\left|w\right|<1.
\end{equation*}
Then we get the Lemma for $a=2$. And it reduces to Lemma 3.2 in \cite{CR} when $z=2t,\;t>0$ in this case.

For the general case of $a>0$,
change variables
$$r=\left(\frac a2\right)^{1/a}r_1^{2/a},\;\;s=\left(\frac a2\right)^{1/a}s_1^{2/a}.$$
Then we get  
\begin{align*}\mathcal T_r^{a,\alpha}q_{a,\alpha;z}(s)&=\left(\frac2a\right)^\alpha\mathcal T_{r_1}^\alpha q_{2,z;\alpha}(s_1)\\&=\left(\frac2a\right)^\alpha\frac{e^{-\frac{\coth\;z}2(r_1^2+s_1^2)}}{(r_1s_1)^\alpha\sinh\;z}I_\alpha\left(\frac{r_1s_1}{\sinh\;z}\right)=\frac{e^{-\frac{\coth\;z}2\frac2a(r^a+s^a)}}{(r^\frac a2s^\frac a2)^\alpha\sinh\;z}I_\alpha\left(\frac{\frac2ar^\frac a2s^\frac a2}{\sinh\;z}\right).\end{align*}
The proof of Lemma 3.4 is therefore completed. This Lemma can also be deduced from Hille--Hardy identity directly.
\end{proof}

Let $z=i\frac{\mathrm\pi}2$. Then from formula (5.7.4) in \cite{L},
$$I_\alpha\left(\frac{\frac2ar^\frac a2s^\frac a2}{\sinh i\frac{\mathrm\pi}2}\right)=e^{-\alpha\pi i/2}J_\alpha\left(\frac2ar^\frac a2s^\frac a2\right)=e^{-\alpha\pi i/2}\frac{\left(\frac2ar^\frac a2s^\frac a2\right)^\alpha}{2^\alpha\Gamma(\al+1)}j_\alpha\left(\frac2ar^\frac a2s^\frac a2\right).$$
So
\begin{align*}I_{a,\alpha;i\frac{\mathrm\pi}{2}}f(r)&=f\ast_{a,\alpha}q_{a,\alpha;i\frac{\mathrm\pi}{2}}(r)=\int_0^\infty f(s)\mathcal T_r^{a,\alpha}q_{a,\alpha;i\frac{\mathrm\pi}{2}}(s)s^{a\alpha+a-1}\,ds\\&=e^{-(\alpha+1)\pi i/2}\frac1{a^\alpha\Gamma(\al+1)}\int_0^\infty f(s)j_\alpha\left(\frac2ar^\frac a2s^\frac a2\right)s^{a\alpha+a-1}\,ds=e^{-(\alpha+1)\pi i/2}H_{a,\alpha}(f)(r).\end{align*}
The proof of Theorem 1.1 is therefore completed. $\hfill\Box$

\section{Proof of Theorem \ref{thm2}}

\

Consider the orthonormal basis \eqref{Y}
of $\mathcal H_k^m\left(\mathbb{R}^N\right)\vert_{\mathbb S^{N-1}}$. Accordingly, we have the $h$-harmonic expansion
for $f\in
L^2\left(\mathbb{R}^N,\vartheta_{k,a}\left(x\right)dx\right)$,
\begin{equation}\label{expansion}f\left(rx'\right)=\sum_{m=0}^\infty\sum_{i=1}^{d(m)}f_{m,i}(r)Y_i^m(x'),\end{equation}
where
$$f_{m,i}(r)=\int_{\mathbb S^{N-1}}f\left(rx'\right)Y_i^m(x'){\vartheta_{k,a}}(x')d\sigma(x').$$

In \cite{BSK2}, the authors proved that  $\Phi _{l,m,j}^{(a)}$ (see \eqref{Phi}) are
eigenfunctions for $-{{\Delta }_{k,a}}$
by interpreting ${{\Delta }_{k,a}}$ in the
framework of the (infinite dimensional) representation of the Lie
algebra $sl(2,\mathbb R)$ (see (3.9 a) and (3.32 a) in \cite{BSK2}) on its dense domain $W_{k,a}\left(\mathbb{R}^N\right)$, i.e.
\begin{equation}\label{Delta}-{{\Delta }_{k,a}}\Phi _{l,m,j}^{(a)}\left( x \right)=a\left( 2l+{{\lambda }_{k,a,m}}+1 \right)\Phi _{l,m,j}^{(a)}\left( x \right).\end{equation}

Then we have the spectral decomposition of $\mathcal I_{k,a}(z)f$ via the basis \eqref{basis4} in terms of Laguerre polynomials for $f\in L^2\left(\mathbb{R}^N,\vartheta_{k,a}\left(x\right)dx\right)$,
\begin{equation}\label{T}\mathcal I_{k,a}(z)f\left(x\right)=\sum_{l,m,j}e^{-z\left(2l+\lambda_{k,a,m}+1\right)}\left\langle f,\Phi_{l,m,j}^{(a)}\right\rangle_{k,a}\Phi_{l,m,j}^{(a)}\left(x\right),\end{equation}
where ${\left\langle f,g\right\rangle}_{k,a}=\int_{\mathbb{R}^N}f(x)g(x)\vartheta_{k,a}(x)dx$.

\noindent{\it Proof of Theorem 1.2}.
By Lemma 3.1, the $a$-deformed Laguerre holomorphic semigroup can also be written as
$$I_{a,\alpha;z}f=\sum_{l=0}^\infty e^{-z\left(2l+\alpha+1\right)}{\left\langle f,\widetilde\varphi_l^{a,\alpha}\right\rangle}_{d\mu_{a,\alpha}}\widetilde\varphi_l^{a,\alpha}.$$
We then apply the spherical harmonic expansion \eqref{expansion} to the
spectral definition \eqref{T} of $T_t^{k,a}f\left(x\right)$. By 
\eqref{Phi} and by noticing that
\begin{equation*}\label{relation}\widetilde{\varphi}_l^{a,\lambda_{k,a,m}}(r)=r^{-m}\psi_{l,m}^{(a)}(r)\end{equation*}
when $\lambda_{k,a,m}\geq -1/2$, we have
\begin{align*}\mathcal I_{k,a}(z)f\left(x\right)&=\sum_{l,m,j}e^{-z\left(2l+\lambda_{k,a,m}+1\right)}\left\langle f,\Phi_{l,m,j}^{(a)}\right\rangle_{k,a}\Phi_{l,m,j}^{(a)}\left(x\right)\\&=\sum_{m=0}^\infty\sum_{j=0}^{d(m)}\sum_{l=0}^\infty\int_0^\infty f_{m,j}\left(r\right)\psi_{l,m}^{(a)}\left(r\right)r^{2\left\langle k\right\rangle+N+a-3}dr\cdot\\&\;\;\;\;e^{-z\left(2l+\lambda_{k,a,m}+1\right)}\psi_{l,m}^{(a)}\left(r\right)Y_{m,j}(x')\\&=\sum_{m=0}^\infty\sum_{j=0}^{d(m)}\sum_{l=0}^\infty\int_0^\infty f_{m,j}\left(r\right)r^{-m}\widetilde{\varphi}_l^{a,\lambda_{k,a,m}}\left(r\right)r^{a\lambda_{k,a,m}+a-1}dr\cdot\\&\;\;\;\;e^{-z\left(2l+\lambda_{k,a,m}+1\right)}\widetilde{\varphi}_l^{a,\lambda_{k,a,m}}\left(r\right)r^mY_{m,j}(x')\\&=\sum_{m,j}Y_{m,j}(x')r^mI_{a,\lambda_{k,a,m};z}\left(\left(\cdot\right)^{-m}f_{m,j}\right)\left(r\right).
\end{align*}

For $f(x)=Y_{m,j}\left(x'\right)\psi(r),\;\psi(r)\in L^2\left({\mathbb{R}}_+,r^{\;2\left\langle k\right\rangle+N+a-3}dr\right),\;x=rx',$ we have the following Hecke-Bochner identity for the $(k,a)$-generalized Laguerre semigroups,
$$\mathcal I_{k,a}(z)f\left(x\right)=Y_{m,j}(x')r^mI_{a,\lambda_{k,a,m};z}\left(\left(\cdot\right)^{-m}\psi\right)\left(r\right).$$
Taking $m=0$, we get the special case for radial functions. The proof of Theorem 1.2 is therefore completed.
$\hfill\Box$

Define the $a$-deformed Dunkl--Hermite heat semigroup with infinitesimal generator ${{\Delta }_{k,a}}$ as
$T_t^{k,a}f:={\mathcal I}_{k,a}\left(ta\right)f,\;t>0$ and the $a$-deformed Laguerre heat semigroup as $T_{a,\alpha;t}f:=I_{a,\alpha;ta}f,\;t>0$.
Then from Theorem 4.1,
\begin{equation}\label{Ttka}T_t^{k,a}f\left(x\right)=\sum_{m,j}Y_{m,j}(x')r^mT_{a,\lambda_{k,a,m},t}\left(\left(\cdot\right)^{-m}f_{m,j}\right)\left(r\right).\end{equation}
It reduces to the equation in Theorem 4.5 in \cite{CR} when $a=2$.

\begin{rem}
The case of $a=2$ of the above argument gives a new proof of the Theorem 4.5 in \cite{CR}. In \cite{CR} the authors proved the Theorem 4.5 by using Dunkl--Hermite expansions and proving the identity for Dunkl--Hermite projections first. But if we use the basis given in terms of Laguerre polynomials, which are also the eigenfunctions of Dunkl Hermite operators, the theorem can be proven directly from the above. For radial functions it was shown in \cite{T2} in classical case that Hermite expansions reduce to Laguerre expansions.
The Heisenberg uncertainty principle for Dunkl transforms was
also proved using the two different expansions successively. It
was first proved by R\"osler using Dunkl--Hermite expansions (see
\cite{R}), and was then proved in \cite[Section 5.7]{BSK2} using the tools we refer to in this paper as well (see \cite{S} also for a proof using the basis given by Dunkl \cite{Du2} in terms of Laguerre polynomials).
\end{rem}

Now we use the following Lemma (see \cite{CR}) to give the expansion of the fractional $(k,a)$-generalized harmonic ocillator into fractional $a$-deformed Laguerre operator (there is a constant missed in \cite[Lemma 3.4]{CR}. Here we give the corrected Lemma).
\begin{lem}(\cite[Lemma 3.4]{CR})
Let $0<\sigma<1$, and $\lambda\in\mathbb{R}$ such that $\lambda+\sigma>-1$. Then,
$$2^\sigma\left|\Gamma\left(-\sigma\right)\right|\frac{\Gamma\left(\frac\lambda2+\frac{1+\sigma}2\right)}{\Gamma\left(\frac\lambda2+\frac{1-\sigma}2\right)}=\int_0^\infty\left(\cosh\;t-1\right)\left(\sinh\;t\right)^{-\sigma-1}dt+\int_0^\infty\left(1-e^{-t\lambda}\right)\left(\sinh\;t\right)^{-\sigma-1}dt.
$$
\end{lem}
Denote by $E_\sigma:=\frac{a^\sigma}{\left|\Gamma\left(-\sigma\right)\right|}\int_0^\infty\left(\cosh\;t-1\right)\left(\sinh\;t\right)^{-\sigma-1}dt.$ Then 
\begin{align*}L_{a,\alpha;\sigma}f\left(r\right)&=\frac{2^{\alpha+1}}{a^\alpha\Gamma(\alpha+1)}\sum_{l=0}^\infty(2a)^\sigma S_l^{a,\alpha;\sigma}f\ast_{a,\alpha}\varphi_l^{a,\alpha}(r)\\&=E_\sigma f\left(r\right)+\frac{a^\sigma}{\left|\Gamma\left(-\sigma\right)\right|}\int_0^\infty\left(f\left(r\right)-T_{a,\alpha;t/a}f\left(r\right)\right)\left(\sinh\;t\right)^{-\sigma-1}dt.\end{align*}
Given $0<\sigma<1$, we define conformally invariant fractional $(k,a)$-generalized harmonic ocillator $\left(-\triangle_{k,a}\right)_{\sigma}$ to be the
operator
$$\left(-\triangle_{k,a}\right)_{\sigma}=(2a)^\sigma\frac{\Gamma\left(\frac{-\Delta_{k,a}}{2a}+\frac{1+\sigma}2\right)}{\Gamma\left(\frac{-\Delta_{k,a}}{2a}+\frac{1-\sigma}2\right)}.$$
So, in view of \eqref{Delta}, $\left(-\triangle_{k,a}\right)_{\sigma}$ corresponds to the spectral multiplier $(2a)^\sigma\Gamma\left(\frac{2l+\lambda_{k,a,m}+1}2+\frac{1+\sigma}2\right)/\\
\Gamma\left(\frac{2l+\lambda_{k,a,m}+1}2+\frac{1-\sigma}2\right)$ and it equals to 
$$E_\sigma f\left(x\right)+\frac{a^\sigma}{\left|\Gamma\left(-\sigma\right)\right|}\int_0^\infty\left(f\left(x\right)-T_{t/a}^{k,a}f\left(x\right)\right)\left(\sinh\;t\right)^{-\sigma-1}dt$$
from Lemma 4.2.
For $a=2$, it should coincide with the fractional Dunkl--Hermite operator in \cite{CR} (there is a constant factor missed in the definition given in \cite {CR}).

By formula \eqref{Ttka},
\begin{align*}\left(-\triangle_{k,a}\right)_{\sigma}f\left(x\right)&=E_\sigma f\left(x\right)+\frac{a^\sigma}{\left|\Gamma\left(-\sigma\right)\right|}\int_0^\infty\left(f\left(x\right)-T_{t/a}^{k,a}f\left(x\right)\right)\left(\sinh\;t\right)^{-\sigma-1}dt\\&=\sum_{m,j}Y_{m,j}(x')r^m\bigg[E_\sigma r^{-m}f_{m,j}\left(r\right)\\&\;\;\;+\frac{a^\sigma}{\left|\Gamma\left(-\sigma\right)\right|}\int_0^\infty\left(r^{-m}f_{m,j}\left(r\right)-T_{a,\lambda_{k,a,m},{t/a}}\left(\left(\cdot\right)^{-m}f_{m,j}\right)\left(r\right)\right)\left(\sinh\;t\right)^{-\sigma-1}dt\bigg]\\&=\sum_{m,j}Y_{m,j}(x')r^m\bigg[E_\sigma g_{m,j}\left(r\right)\\&\;\;\;+\frac{a^\sigma}{\left|\Gamma\left(-\sigma\right)\right|}\int_0^\infty\left(g_{m,j}\left(r\right)-T_{a,\lambda_{k,a,m},{t/a}}g_{m,j}\left(r\right)\right)\left(\sinh\;t\right)^{-\sigma-1}dt\bigg]\\&=\sum_{m,j}Y_{m,j}(x')r^mL_{a,\lambda_{k,a,m};\sigma}g_{m,j}\left(r\right),\end{align*}
where $g_{m,j}\left(r\right)=r^{-m}f_{m,j}\left(r\right)$.

The following Lemma was found by Yafaev \cite{Y} for $v=m/2,\;m\in \mathbb N$, and was then proved in \cite{GI2} for any $v>0$.
\begin{lem}(\cite[Lemma 2.3]{GI2})
If $v>0$, then $$\frac{\Gamma\left(t+v\right)}{\Gamma\left(\tau+v\right)}<\frac{\Gamma\left(t\right)}{\Gamma\left(\tau\right)},\;\;\;\;\;\;0<t<\tau.$$
\end{lem}

By Theorem 3.3 we have 
\begin{align*}{\left\langle\left(-\triangle_{k,a}\right)_{\sigma}f,f\right\rangle}&_{L^2\left(\mathbb{R}^N,\vartheta_{k,a}(x)dx\right)}=\sum_{m=0}^\infty\sum_{j=1}^{d(m)}{\left\langle L_{a,\lambda_{k,a,m};\sigma}g_{m,j},g_{m,j}\right\rangle}_{L^2\left(\left(0,\infty\right),d\mu_{a,\lambda_{k,a,m}}\left(r\right)\right)}\\&\geq\sum_{m=0}^\infty\sum_{j=1}^{d(m)}\left(\frac{2a}\delta\right)^\sigma\left(B_{\lambda_{k,a,m},\sigma}^\delta\right)^2\int_0^\infty\left|g_{m,j}(r)\right|^2\frac{\omega_{\lambda_{k,a,m},\sigma}^{\delta,a}(r)}{\omega_{\lambda_{k,a,m},-\sigma}^{\delta,a}(r)}d\mu_{a,\lambda_{k,a,m}}(r)\\&=\sum_{m=0}^\infty\sum_{j=1}^{d(m)}\left(\frac{2a}\delta\right)^\sigma\left(B_{\lambda_{k,a,m},\sigma}^\delta\right)^2\int_0^\infty\left|f_{m,j}(r)\right|^2\frac{\omega_{\lambda_{k,a,m},\sigma}^{\delta,a}(r)}{\omega_{\lambda_{k,a,m},-\sigma}^{\delta,a}(r)}d\mu_{a,\lambda_a}(r).\end{align*}
Then by Lemma 4.3 and a similar argument as in the end of the proof in \cite{CR}, 
\begin{align*}\left(\frac{2a}\delta\right)^\sigma\left(B_{\lambda_{k,a,m},\sigma}^\delta\right)^2\frac{\omega_{\lambda_{k,a,m},\sigma}^{\delta,a}(r)}{\omega_{\lambda_{k,a,m},-\sigma}^{\delta,a}(r)}&=\left(\frac a2\right)^\sigma\delta^\sigma\frac{\Gamma\left(\frac{\lambda_{k,a,m}+2+\sigma}2\right)}{\Gamma\left(\frac{\lambda_{k,a,m}+2-\sigma}2\right)}\frac{K_{(\lambda_{k,a,m}+1+\sigma)/2}\left((\delta+\frac2ar^a)/2\right)}{K_{(\lambda_{k,a,m}+1-\sigma)/2}\left((\delta+\frac2ar^a)/2\right)}\left(\delta+{\textstyle\frac2a}r^a\right)^{-\sigma}\\&\geq\left(\frac a2\right)^\sigma\delta^\sigma\frac{\Gamma\left(\frac{\lambda_a+2+\sigma}2\right)}{\Gamma\left(\frac{\lambda_a+2-\sigma}2\right)}\left(\delta+{\textstyle\frac2a}r^a\right)^{-\sigma}=\left(\frac a2\right)^\sigma B_{\lambda_a,\sigma}^\delta\left(\delta+{\textstyle\frac2a}r^a\right)^{-\sigma}.\end{align*}
Therefore,
\begin{align*}
{\left\langle\left(-\triangle_{k,a}\right)_{\sigma}f,f\right\rangle}_{L^2\left(\mathbb{R}^N,\vartheta_{k,a}(x)dx\right)}&\geq\sum_{m,j}\left(\frac{2a}\delta\right)^\sigma\left(B_{\lambda_{k,a,m},\sigma}^\delta\right)^2\int_0^\infty\left|f_{m,j}(r)\right|^2\frac{\omega_{\lambda_{k,a,m},\sigma}^{\delta,a}(r)}{\omega_{\lambda_{k,a,m},-\sigma}^{\delta,a}(r)}d\mu_{a,\lambda_a}(r)\\&\geq \left(\frac a2\right)^\sigma B_{\lambda_a,\sigma}^\delta\sum_{m,j}\int_0^\infty\left|f_{m,j}(r)\right|^2\left(\delta+{\textstyle\frac2a}r^a\right)^{-\sigma}d\mu_{a,\lambda_a}(r)\\&=\left(\frac a2\right)^\sigma B_{\lambda_a,\sigma}^\delta\int_{\mathbb{R}^N}\frac{\left|f(x)\right|^2}{\left(\delta+\frac2a\left|x\right|^a\right)^\sigma}\vartheta_{k,a}\left(x\right)dx.
\end{align*}
The proof of Theorem 1.4 is completed.

\section{Characterization of $\Omega_{k,a}^{}\left(\gamma_z\right)$ on each $k$-spherical component $\mathcal H_k^m\left(\mathbb{R}^N\right)$}

\

In Section 4.1 of \cite{BSK2} the authors gave the definition of the radial part $\Omega_{k,a}^{(m)}\left(\gamma_z\right)$, $\Re z\geq 0$ of the holomorphic semigroup  $\Omega_{k,a}\left(\gamma_z\right)={\mathcal I}_{k,a}\left(z\right)$ on each $k$-spherical component $\mathcal H_k^m\left(\mathbb{R}^N\right)$ via a decomposition of unitary representation  $\Omega_{k,a}\left(\gamma_z\right)$ of the universal covering group $S\widetilde{L(2,}\mathbb R)$ on $L^2\left(\mathbb{R}^N\vartheta_{k,a}(x)dx\right)$ (see \cite[Section 4.1]{BSK2} for the detailed definition of $\Omega_{k,a}^{(m)}\left(\gamma_z\right)$). And they showed that the unitary operator $\Omega_{k,a}^{(m)}\left(\gamma_z\right)$ on $L^2\left({\mathbb{R}}_+,r^{\;2\left\langle k\right\rangle+N+a-3}dr\right)$ can be expressed as 
\begin{align}\label{Omega}\Omega_{k,a}^{(m)}\left(\gamma_z\right)f(r)=\int_0^\infty\Lambda_{k,a}^{(m)}\left(r,s;z\right)f(s)s^{2\left\langle k\right\rangle+N+a-3}ds,
\end{align}
where $\Lambda_{k,a}^{(m)}\left(r,s;z\right)$ has its closed formula (see \cite[(4.11)]{BSK2})
$$\Lambda_{k,a}^{(m)}\left(r,s;z\right)=\frac{\left(rs\right)^{-\left\langle k\right\rangle-\frac N2+1}}{\sinh\;z}e^{-\frac{\coth z}a(r^a+s^a)}I_{\lambda_{k,a,m}}\left(\frac{\frac2ar^\frac a2s^\frac a2}{\sinh z}\right).$$
The integral on the right hand side of \eqref{Omega} converges for $f\in L^2\left({\mathbb{R}}_+,r^{\;2\left\langle k\right\rangle+N+a-3}dr\right)$ if $\Re z>0$ and for all $f$ in the dense subspace of $L^2\left({\mathbb{R}}_+,r^{\;2\left\langle k\right\rangle+N+a-3}dr\right)$ spanned by the functions $\left\{\psi_{l,m}^{(a)}(r):\;l\in\mathbb{N}\right\}$ if $\Re z=0$ (see \eqref{psi} for the definition of $\psi_{l,m}^{(a)}(r)$).
We give an explicit expression of $\Omega_{k,a}^{(m)}\left(\gamma_z\right)$ in this section via the $a$-deformed Laguerre operator $L_{a,\alpha}$ (see \cite{B} for the case of $a=2$ on such expression).

\begin{thm}
For $\lambda_{k,a,m}\geq -1/2$, $\Re z\geq 0$ and $s>0$, $\Omega_{k,a}^{(m)}\left(\gamma_z\right)$ acting on $L^2\left({\mathbb{R}}_+,r^{\;2\left\langle k\right\rangle+N+a-3}dr\right)$ has the form
$$\Omega_{k,a}^{(m)}\left(\gamma_z\right)f(s)=s^{m}I_{a,\lambda_{k,a,m};z}\left(\left(\cdot\right)^{-m}f\right)\left(s\right),\;f\in L^2\left({\mathbb{R}}_+,r^{\;2\left\langle k\right\rangle+N+a-3}dr\right).$$
Thus
$${\left.\frac d{dz}\right|}_{z=0}\Omega_{k,a}^{(m)}\left(\gamma_z\right)f(s)=-s^m\frac1aL_{a,\lambda_{k,a,m}}\left(\left(\cdot\right)^{-m}f\right)\left(s\right).$$
\end{thm}

\begin{proof}
We can take $\al$ as $\lambda_{k,a,m}$ in Lemma 3.4, then we get 
\begin{align}\label{Tr}\mathcal T_r^{a,\lambda_{k,a,m}}q_{a,\lambda_{k,a,m};z}(s)=(rs)^{-m}\Lambda_{k,a}^{(m)}\left(r,s;z\right).\end{align}

For every $f$ in the dense subspace of $L^2\left({\mathbb{R}}_+,r^{\;2\left\langle k\right\rangle+N+a-3}dr\right)$ spanned by the functions $\left\{\psi_{l,m}^{(a)}(r):\;l\in\mathbb{N}\right\}$, we have
\begin{align*}\Omega_{k,a}^{(m)}\left(\gamma_z\right)f(s)&=\int_0^\infty f(r)\Lambda_{k,a}^{(m)}\left(r,s;z\right)r^{\;2\left\langle k\right\rangle+N+a-3}dr\\&=s^{m}\int_0^\infty r^{-m}f(r)\mathcal T_r^{a,\lambda_{k,a,m}}q_{a,\lambda_{k,a,m};z}(s)r^{\;2m+2\left\langle k\right\rangle+N+a-3}dr\\&=s^{m}I_{a,\lambda_{k,a,m};z}\left(\left(\cdot\right)^{-m}f\right)\left(s\right).
\end{align*}
Then from the boundedness of the operator $\Omega_{k,a}^{(m)}\left(\gamma_z\right)$ on $L^2\left({\mathbb{R}}_+,r^{\;2\left\langle k\right\rangle+N+a-3}dr\right)$, we get 
$$\Omega_{k,a}^{(m)}\left(\gamma_z\right)f(s)=s^{m}I_{a,\lambda_{k,a,m};z}\left(\left(\cdot\right)^{-m}f\right)\left(s\right)$$ for all $f\in L^2\left({\mathbb{R}}_+,r^{\;2\left\langle k\right\rangle+N+a-3}dr\right)$.
\end{proof}
\begin{rem}
i). From this theorem, the spherical harmonic expansion of the $(k,a)$-generalized Laguerre semigroup \eqref{expansion2} can be derived directly by (4.3) in \cite{BSK2}.\\
ii). Taking $m=0$, we get the formula of $\triangle_{k,a}$ on radial Schwartz functions $f=f_0\left(\left|\cdot\right|\right)$, $f_0\in \mathcal S\left({\mathbb{R}}_+\right)$,
$$\triangle_{k,a}f(x)=-L_{a,\lambda_a}\left(f_0\right)\left(r\right),\;r=\left|x\right|.$$
This is equivalent to the formula of Dunkl Laplacian $\triangle_{k}$ on radial functions in \cite[Proposition 4.15]{M}.
\end{rem}

\

\section*{Acknowledgments}The author would like to thank the referee and thank also Salem Ben Sa\"id, Toshiyuki Kobayashi and his adviser Nobukazu Shimeno very much for insightful comments. Furthermore, the author would like to thank Luz Roncal for helpful discussions.

\


\begin{thebibliography}{99}

\bibitem{A} R. Askey, \newblock {\em Orthogonal polynomials and positivity}, In: Studies in Applied
Mathematics, Wave propagation and special functions, SIAM (1970), 64--85.

\bibitem{B} S. Ben Sa\"id, \newblock {\em On the integrability of a representation of $sl(2,\mathbb R)$}, \newblock J. Funct. Anal., {\bf 250}(2007), 249--264

\bibitem{BSK1} S. Ben Sa\"id, T. Kobayashi, B. \O rsted, \newblock {\em Generalized Fourier transforms ${{F}_{k,a}}$}, C. R. Math. Acad. Sci. Paris,  347[19-20](2009), 1119--1124.

\bibitem{BSK2} S. Ben Sa\"id, T. Kobayashi, B. \O rsted, \newblock {\em Laguerre semigroup and Dunkl operators}, Compos. Math., 148[4](2012), 1265--1336.

\bibitem{Br} T.P. Branson, L. Fontana and C. Morpurgo, \newblock {\em Moser-Trudinger and Beckner-Onofri's inequalities on the CR sphere}, Ann. Math. 177 (2013), 1--52.

\bibitem{CR} \'O. Ciaurri, L. Roncal, S. Thangavelu, \newblock {\em Hardy-type inequalities for fractional powers of the Dunkl--Hermite operator}, Proc. Edinburg Math. Soc.(2018), 1--32.

\bibitem{Du1} C.F. Dunkl, \newblock {\em Differential-difference operators associated to reflection groups}, Trans. Amer. Math. Soc., 311, no. 1(1989), 167--183.

\bibitem{Du2} C.F. Dunkl, \newblock {\em Hankel transforms associated to finite reflection groups, Hypergeometric Functions on Domains of Positivity, Jack Polynomials, and Applications}, Proceedings of an AMS Special Session Held March 22-23, 1991 in Tampa, Florida (Vol. 138, p. 123). American Mathematical Soc..

\bibitem{GI1} D. Gorbachev, V. Ivanov, S. Tikhonov, \newblock {\em Sharp Pitt inequality and logarithmic uncertainty principle for Dunkl transform in $L^2$}, Journal of Approximation Theory, (2016), 109--118.

\bibitem{GI2} D. Gorbachev, V. Ivanov, S. Tikhonov, \newblock {\em Pitt's inequalities and uncertainty principle for generalized Fourier transform}, International Mathematics Research Notices,  Issue 23(2016), 7179--7200.

\bibitem{H} R. Howe, \newblock {\em The oscillator semigroup. The mathematical heritage of Hermann Weyl (Durham, NC, 1987)}, 61--132, Proc. Sympos. Pure Math., 48, Amer. Math. Soc., Providence, RI, 1988.

\bibitem{KM1} T. Kobayashi, G. Mano, \newblock {\em  The inversion formula and holomorphic extension of the minimal repre-
sentation of the conformal group}, Harmonic Analysis, Group Representations, Automorphic Forms and
Invariant Theory: In honor of Roger Howe, (eds. J. S. Li, E. C. Tan, N. Wallach and C. B. Zhu), World
Scientific(2007), 159--223,.

\bibitem{KM2} T. Kobayashi, G. Mano, \newblock {\em  The Schr\"odinger model for the minimal representation of the indefinite
orthogonal group O(p,q)}, Mem. Amer. Math. Soc. 213(1000), 2011.

\bibitem{L} N. N. Lebedev, \newblock {\em Special functions and its applications}, Dover, New York, 1972.

\bibitem{M} H. Mejjaoli, K. Trim\`eche, \newblock {\em  On a mean value property associated with the dunkl laplacian operator and applications, Integral Transforms and Special Functions}, 12:3(2001), 279--302, DOI: 10.1080/10652460108819351.

\bibitem{Mc} J. McCully, \newblock {\em The Laguerre transform}, SIAM Rev. 2 (1960), 185--191.

\bibitem{Ron} L. Roncal, S. Thangavelu, \newblock {\em Hardy’s inequality for fractional powers of the sublaplacian on the Heisenberg group}, Adv. Math. 302 (2016), 106--158.

\bibitem{R} M. R\"osler, \newblock {\em An uncertainty principle for the Dunkl transform}, Bull. Austral. Math. Soc., 59(1999), 353--360.

\bibitem{R1} M. R\"osler, \newblock {\em Generalized Hermite polynomials and the heat equation for Dunkl operators}, Comm. Math.
Phys., 192(3)(1998), 519--542.

\bibitem{S} N. Shimeno, \newblock {\em A Note on the Uncertainty Principle for the Dunkl Transform}, Journal of Mathematical Sciences (University of Tokyo), 8(1)(2001), 33--42.

\bibitem{T1} S. Thangavelu, \newblock {\em Lectures on Hermite and Laguerre expansions}, Mathematical Notes 42. Princeton University Press, Princeton, NJ, 1993.

\bibitem{T2} S. Thangavelu, \newblock {\em Hermite and Laguerre semigroups: some recent developments}, CIMPA lecture notes, 2006.

\bibitem{Y} D. Yafaev, \newblock {\em Sharp constants in the Hardy-Rellich inequalities}, J. Funct. Anal., 168(1999), 121--144.

\end{thebibliography}
\end{document}